\documentclass[12pt]{amsart}
\usepackage{amsmath,latexsym,amssymb,amsfonts,stmaryrd}

\usepackage{a4wide}

\newcommand{\LRmonotone}{anti-increasing}
\newcommand{\LRorder}{anti-increasing order}
\newcommand{\LRordered}{anti-increasingly ordered}
\newcommand{\LRmonotonelabeled}{anti-increasingly labeled}

\usepackage{tikz}
\usetikzlibrary{positioning}

\tikzset{
  every text node part/.style = {font=\scriptsize},
  grow=up,
  inner sep=0.2mm,
  vertex/.style={circle,draw},
  emptyvertex/.style={},
  level distance=3ex,
  level 1/.style={sibling distance=4em},
  level 2/.style={sibling distance=2em}
}

\newsavebox\wrapasybox
\newlength\boxh

\newcommand{\treebox}[1]{\raisebox{0.5ex}{\raisebox{-0.5\height}{\usebox{#1}}}}

\newenvironment{mytikzbox}{
\begin{lrbox}{\wrapasybox}
\begin{tikzpicture}
}{
\end{tikzpicture}
\end{lrbox}\treebox{\wrapasybox}
}

\usepackage{graphicx}
\theoremstyle{plain}
\newtheorem{theorem}{Theorem}
\newtheorem{lemma}[theorem]{Lemma}
\newtheorem{prop}[theorem]{Proposition}
\newtheorem{corollary}[theorem]{Corollary}
\theoremstyle{definition}
\newtheorem{definition}[theorem]{Definition}
\newtheorem{remark}[theorem]{Remark}

\numberwithin{equation}{section}
\numberwithin{theorem}{section}

\DeclareMathOperator{\IE}{\mathbf{E}}
\newcommand{\ET}{|} 
\newcommand{\abs}[1]{\left\lvert #1 \right\rvert}  
\DeclareMathOperator{\DeltaBF}{\Delta^\gamma}
\DeclareMathOperator{\ip}{\mathrm{ip}}

\newcommand{\alg}[1]{\mathcal{#1}}

\newcommand{\cc}{\kappa} 
\newcommand{\fc}{c} 
\newcommand{\bc}{h} 

\newcommand{\SP}{\cP} 

\newcommand{\intpart}{\alg{I}}  
\newcommand{\SPconn}{\SP^{conn}}  
\newcommand{\SPirr}{\SP^{irr}}  
\newcommand{\NC}{NC}  
\newcommand{\NCirr}{\NC^{irr}}  

\newcommand{\PRBT}{\mathcal{PRBT}}
\newcommand{\PRT}{\mathcal{PRT}}

\def\ff{\varphi}
\def\NN{\mathbb{N}}
\def\RR{\mathbb{R}}

\def\HH{\mathcal{H}}
\def\cP{\mathcal{P}}

\title{The normal distribution is $\boxplus$-infinitely divisible}

\author[Belinschi]{Serban T. Belinschi}
\address{Department of Mathematics and Statistics, \\
University of Saskatchewan,\\
and Institute of Mathematics of the Romanian Academy;\\
106 Wiggins Road, Saskatoon, SK, S7N 5E6, Canada
}
\email{belinschi@math.usask.ca}
\thanks{Work of S. Belinschi was supported in part
by a Discovery Grant of the Natural
Sciences and Engineering Research Council of Canada and a University of
Saskatchewan start-up grant.}

\author[Bo\.{z}ejko]{Marek Bo\.{z}ejko}
\address{Instytut Matematyczny, Uniwersytet Wroc{\l}awski,\\
 Pl.\ Grunwaldzki 2/4, 50-384 Wroc{\l}aw, Poland}
\email{bozejko@math.uni.wroc.pl}
\thanks{Work of M.Bo\.zejko was partially supported by Grant N N 201 364436 of Polish Ministry of Science
}

\author[Lehner]{Franz Lehner}
\address{Institute for  Mathematical  Structure Theory,\\
Graz Technical University,\\
Steyrergasse 30, 8010 Graz, Austria}
\email{lehner@finanz.math.tu-graz.ac.at}

\author[Speicher]{Roland Speicher}
\address{Queen's University, Department of Mathematics and Statistics,\\
Jeffery Hall, Kingston, ON K7L 3N6, Canada}
\email{speicher@mast.queensu.ca}
\thanks{Work of R. Speicher was supported by a Discovery Grant from NSERC}

\date{\today}

\begin{document}

\begin{abstract}
We prove that the classical normal distribution is infinitely divisible
with respect to the free additive convolution.
We study the Voiculescu transform first by giving a survey
of its combinatorial implications and then analytically,
including a proof of free infinite divisibility.
In fact we prove that a subfamily Askey-Wimp-Kerov distributions
are freely infinitely divisible, of which the normal distribution
is a special case.
At the time of this writing this is only the third example known to us
of a nontrivial distribution that is infinitely divisible
with respect to both classical and free convolution,
the others being the Cauchy distribution
and the free $1/2$-stable distribution.
\end{abstract}

\maketitle

\section{Introduction}

We will prove that the classical normal distribution is infinitely divisible with respect
to free additive convolution.

This fact might come as a surprise, since the classical Gaussian distribution has no
special role in free probability theory. The first known explicit mentioning of that
possibility to one of us was by Perez-Abreu at a meeting in Guanajuato in 2007.
This conjecture had arisen out of joint work with Arizmendi \cite{APA}.

Later when the last three of the present authors met in Bielefeld in the fall of 2008
they were led to reconsider this question in the context of investigations about general
Brownian motions. We want to give in the following some kind of context for this.

In \cite{BSp2} two of the present authors were introducing the class of generalized
Brownian Motions (GBM), i.e., families of self-adjoint operators $G(f)$ ($f\in\HH$, for
some real Hilbert space $\HH$) and a state $\ff$ on the algebra generated by the $G(f)$,
given by
$$\ff(G(f_{1})...G(f_{2n})) = \sum_{\pi\in \SP_{2}(2n)} t(\pi)
\prod_{(i,j)\in\pi}\langle f_{i},f_{j}\rangle.$$
Here $\SP_2(2n)$ denotes the set of pairings of $2n$ elements, and $t$ is a weight
function for such pairings. The concrete form of $t$ determines the specific Brownian
motion.

The most natural example for such a GBM is classical Brownian motion, where $t(\pi)= 1$
for all pairings $\pi$; in this case one gets the normal law for $G(f)$. The $q$-Brownian
motion fits into this frame by putting $t_{q}(\pi) = q^{cr(\pi)}$ (where $cr(\pi)$
denotes the number of crossings of the pairing $\pi$); in this case the law of the random
variable $G(f)$ is related with the theta function of Jacobi, and called $q$-Gaussian
distribution $\gamma_{q}$, see \cite{Ansh,BSp1,BKS}. If the parameter $q$ changes from -1
to 1, one gets an interpolation between the fermionic Brownian Motion ($q=-1$), the free
Brownian Motion ($q=0$), and the classical Brownian Motion ($q=1$).

In \cite{BSp2}, the model of the free product of classical Brownian motions resulted in a
new class of GBM, with the function $t$ given by
\begin{equation}\label{eq:connected}
t_{s}(\pi) = s^{cc(\pi)},
\end{equation}
where $cc(\pi)$ is the number of connected components of the pairing $\pi$. Here $s$ has
to be bigger than 1. This contains as a special case the result: The $2n$-th moment of
the free additive power of the normal law  $\gamma_{1}$ is given as follows:
$$m_{2n} (\gamma_{1}^{\boxplus s}) = \sum_{\pi \in
\SP_{2}(2n)} s^{cc(\pi)},$$ for $s>1$ .

In the light of earlier examples where similar combinatorial identities could be extended
beyond their primary domain of applicability (see, e.g., \cite{Boz}), it was natural to
ask whether this relation could also make any sense for $s<1$. So, in this context, a
natural problem is whether the sequence on the right side of the above formula is a
moment sequence for all $s>0$? This question is equivalent to the free infinite
divisibility of the normal law! One can check easily that the corresponding
more general question on generalized Brownian motions (i.e., whether the $t_s$ from equation
\eqref{eq:connected} is still positive for $s<1$) has a negative answer. From this point
of view, the free infinite divisibility of the classical Gauss seemed quite unlikely.
However, numerical evidence suggested the validity of that conjecture. In this paper we
will give an analytical proof for this conjecture. We want to point out that it still
remains somehow a mystery whether the $\boxplus$-infinite divisibility of the Gauss distribution is
a singular result or whether there is a more conceptual broader theory behind this.

Another example of this phenomenon was found in
\cite{PerezAbreuSakuma:2008:free},
namely that the $1/2$-free stable law
\cite{BercoviciPata:1999:stable}
is also classically infinitely divisible,
being a $\beta$-distribution of the second kind with density
$(4x-1)^{1/2} / x^2$.

\subsection{Related Questions}
One way to describe certain probability distributions on $\mathbb R$
is by specifying their orthogonal polynomials. The orthogonal 
polynomials of the classical Gaussian distribution are the so-called 
Hermite polynomials \cite{Kerov:1998:interlacing}. In \cite{Askey-Wimp}, Askey and 
Wimp describe a family of deformations, indexed by $c\in(-1,+\infty)
$, of the Hermite polynomials, called the associated Hermite 
polynomials. These polynomials are orthogonal with respect to a
family of probability measures $\{\mu_c\colon c\in(-1,+\infty)\}$,
which can be described in terms of a continued fraction expansion
of their Cauchy-Stieltjes transform as
$$
G_{\mu_c}(z)=\frac{1}{z-\displaystyle\frac{c+1}{z-\displaystyle\frac{c+2}{z-\ddots}}}.
$$
The Cauchy-Stieltjes transform of a measure $\mu$
on the real line is defined by 
$$
G_{\mu}(z)=\int_\mathbb R\frac{1}{z-t}\,d\mu(t),\quad z\in\mathbb C^+;
$$
when $\mu$ is a positive measure, 
this function maps $\mathbb C^+$ into the lower half-plane. For 
details we refer to the excellent book of Akhieser \cite{akhieser}.
For $c=0$, we have $\mu_0=\gamma$, the normal distribution 
\cite{Kerov:1998:interlacing}, and one can easily check that one can extend this 
family continuously to $c=-1$ by letting $\mu_{-1}=\delta_0$, the 
probability giving mass one to $\{0\}$. This family, introduced in
\cite{Askey-Wimp}, plays an important role in 
\cite{Kerov:1998:interlacing}; we shall
call its members the {\em Askey-Wimp-Kerov distributions}. 
It will turn out from our proof that $\{\mu_c\colon c\in[-1,0]\}$
are freely infinitely divisible.
Numerical computations show that for several values of $c>0$, $\mu_c$ 
is not freely infinitely divisible. Numerical evidence seems also to indicate that $\mu_c$
is classicaly infinitely divisible only when $c=0$ or $c=-1$.

An interesting interpolation between the normal and the semicircle law was
constructed by Bryc, Dembo and Jiang \cite{BDJ} and further investigated by Buchholz
\cite{Buch}. This leads to a generalized Brownian motion, given by a weight function
($0<b<1$)
$$t_{b}(\pi) = b^{n-h(\pi)},$$
where for a pairing $\pi$ of $2n$ elements we denote by $h(\pi)$ the number of connected
components which have only one block with 2 elements. It was calculated in
\cite{BDJ,Buch}, that the measures satisfy
$$\nu_{b^{2}} = D_b \gamma_{1}\boxplus D_{1-b}\gamma_{0},$$
where $D_{b}$ is the dilation of a measure by parameter $b>0$.
With the results of the present paper it is immediate that
these measures are  infinitely divisible with respect to free convolution.
However a short calculation shows that
these measures are not infinitely divisible with respect to classical
convolution unless $b=1$.

One other tempting example is to consider the distribution of $N\times N$ Gaussian random
matrices. For $N=1$, this is the classical Gauss distribution, whereas for $N\to\infty$
it converges to the semicircle distribution. Both of them are infinitely divisible in the
free sense (the prior by our main result here, and the latter because the semicircle is
the limit in the free central limit theorem). So one might conjecture that the
interpolating distributions, for integer $1<N<\infty$, are also freely infinitely
divisible. However, numerical calculations of the first few Jacobi coefficients of the
corresponding moments, using the Harer-Zagier recurrence, show readily that this is not
the case.

One may also ask the ``opposite'' question, whether
the Wigner distribution is infinitely divisible with respect to classical
convolution.
However this is impossible because any nontrivial classically infinitely divisible
measure has unbounded support, see
\cite[Proposition 2.3]{SteutelHarn:2004:infinite}.

For the same reason the distributions whose density is
a power of the Wigner density
are not classically infinitely divisible.
It is however an open question whether the latter are freely infinitely
divisible. 
Numerical evidence points to a positive answer to this question.
This would provide another
proof that the normal law is freely infinitely divisible. See \cite{APA} for a
survey on these questions.

In the next section we will consider the combinatorial aspects of the
free infinite divisibility of the classical Gaussian distribution; in
particular, we will give some new combinatorial interpretations for the
free cumulants of the Gaussian distribution. In Section 3, we will
then give an analytical proof of our free infinite divisibility
result.

\emph{Acknowledgements} The last three authors would like to thank
Prof.~G\"otze for the kind invitations and
hospitality at SFB 701, Bielefeld. STB would also like to thank
Michael Anshelevich for many useful discussions, especially regarding 
the Riccati equation.

\section{Combinatorial considerations}

\subsection{Partitions}
First we review a few properties of set partitions which will
be needed below. As usual, set partitions will be depicted by
diagrams like the ones shown in Figure~\ref{fig:Partitions}.
\begin{definition}
  A partition of $[n]=\{1,2,\dots,n\}$ is called
  \begin{enumerate}
   \item \emph{connected}
    if no proper subinterval of $[n]$ is a union of blocks;
    this means that any diagram depicting the partition is a connected graph.
   \item  \emph{irreducible}
    if $1$ and $n$ are in the same connected component,
    i.e., there is only one outer block.
   \item \emph{noncrossing}
    if its blocks do not intersect
    in their graphical representation, i.e., if there are no two
    distinct blocks $B_1$ and $B_2$ and elements $a, c\in B_1$ and $b,d\in B_2$
    s.t.\ $a<b<c<d$. Equivalently one could say that a partition is noncrossing
    if each of its connected components consists of exactly one block.
  \end{enumerate}
  Typical examples of these types of partitions are shown in Fig.~\ref{fig:Partitions}.

  We denote the lattice of partitions of $[n]$ by $\SP_{\!\!n}$, the irreducible
  partitions by $\SPirr_{\!\!n}$ and the order ideal of
  connected partitions by $\SPconn_{\!\!n}$; the lattice of noncrossing partitions
  will be denoted by $\NC_n$, and the sublattice of irreducible noncrossing partitions by
  $\NCirr_n$.
  Finally, let us denote by $\intpart_n$ the lattice of interval partitions,
  i.e.\ the lattice of partitions consisting entirely of intervals.
\end{definition}
\begin{center}
  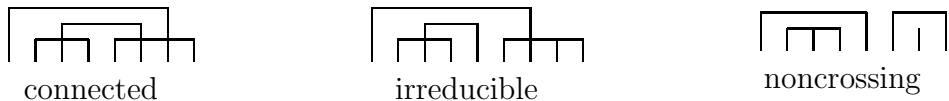
\begin{figure}[htbp]
%
%
    \begin{minipage}{.3\textwidth}
      \begin{center}
        \begin{picture}(80,20.4)(1,0)
          \put(10,0){\line(0,1){20.4}}
          \put(20,0){\line(0,1){8.4}}
          \put(30,0){\line(0,1){14.4}}
          \put(40,0){\line(0,1){8.4}}
          \put(50,0){\line(0,1){8.4}}
          \put(60,0){\line(0,1){14.4}}
          \put(70,0){\line(0,1){20.4}}
          \put(80,0){\line(0,1){8.4}}
          \put(20,8.4){\line(1,0){20}}
          \put(30,14.4){\line(1,0){30}}
          \put(10,20.4){\line(1,0){60}}
          \put(50,8.4){\line(1,0){30}}
        \end{picture}
        
        connected
      \end{center}
    \end{minipage}
%
%
    \begin{minipage}{.3\textwidth}
      \begin{center}
      \begin{picture}(90,20.4)(1,0)
        \put(10,0){\line(0,1){20.4}}
        \put(20,0){\line(0,1){8.4}}
        \put(30,0){\line(0,1){14.4}}
        \put(40,0){\line(0,1){8.4}}
        \put(50,0){\line(0,1){14.4}}
        \put(60,0){\line(0,1){8.4}}
        \put(70,0){\line(0,1){20.4}}
        \put(80,0){\line(0,1){8.4}}
        \put(90,0){\line(0,1){8.4}}
        \put(20,8.4){\line(1,0){20}}
        \put(30,14.4){\line(1,0){20}}
        \put(10,20.4){\line(1,0){60}}
        \put(60,8.4){\line(1,0){30}}
      \end{picture}
        irreducible
      \end{center}
    \end{minipage}
%
%
    \begin{minipage}{.3\textwidth}
      \begin{center}
        \begin{picture}(80,14.4)(1,0)
          \put(10,0){\line(0,1){14.4}}
          \put(20,0){\line(0,1){8.4}}
          \put(30,0){\line(0,1){8.4}}
          \put(40,0){\line(0,1){8.4}}
          \put(50,0){\line(0,1){14.4}}
          \put(60,0){\line(0,1){14.4}}
          \put(70,0){\line(0,1){8.4}}
          \put(80,0){\line(0,1){14.4}}
          \put(20,8.4){\line(1,0){20}}
          \put(10,14.4){\line(1,0){40}}
          \put(70,8.4){\line(1,0){0}}
          \put(60,14.4){\line(1,0){20}}
        \end{picture}
        
        noncrossing
      \end{center}
    \end{minipage}
    \label{fig:Partitions}
    \caption{Typical partitions}
  \end{figure}
\end{center}

\subsection{Cumulants}
Cumulants linearize convolution of probability measures coming from
various notions of independence.
\begin{definition}
  A \emph{non-commutative probability space} is a pair $(\alg{A},\phi)$
  of a (complex) unital algebra $\alg{A}$ and a unital linear functional $\phi$.
  The elements of $\alg{A}$ are called \emph{(non-commutative) random variables}.
\end{definition}
Given a notion of independence, convolution is defined as follows.
Let $a$ and $b$ be ``independent'' random variables,
then the convolution of the distributions of $a$ and $b$ is defined to be
the distribution of the sum $a+b$. In all the examples below, the distribution
of the sum of  ``independent'' random variables only depends on the individual
distributions of the summands and therefore convolution is well defined
on the level of probability measures. 
Moreover,
the $n$-th moment $m_n(a+b)$ is a polynomial function of the moments
of $a$ and $b$ of order less or equal to $n$.
For our purposes it is sufficient to axiomatize cumulants as follows.
\begin{definition}
  \label{def:cumulants}
  Given a notion of independence on a noncommutative probability space $(\alg{A},\phi)$,
  a sequence of maps $a\mapsto k_n(a)$, $n=1,2,\dots$ is called a \emph{cumulant sequence} if it satisfies the following properties
  \begin{enumerate}
   \item $k_n(a)$ is a polynomial in the first $n$ moments of $a$ with leading term $m_n(a)$.
    This ensures that conversely the moments can be recovered from the cumulants.
   \item homogeneity: $k_n(\lambda a) = \lambda^n k_n(a)$.
   \item additivity: if $a$ and $b$ are ``independent'' random variables,
    then $k_n(a+b)=k_n(a)+k_n(b)$.
  \end{enumerate}
\end{definition}
M\"obius inversion on the lattice of partitions plays a crucial role
in the combinatorial approach to cumulants.
We  need three kinds of cumulants here,
corresponding to classical, free and boolean
independence, which involve the three lattices
of set partitions, noncrossing partitions and interval partitions,
respectively.
Let $X$ be a random variable with distribution $\psi$
and moments $m_n = m_n(X) = \int x^n\, d\psi(x)$

\subsection{Classical cumulants}
  Let
  $$
  \alg{F}(z) = \int e^{xz}\,d\psi(x)
       = \sum_{n=0}^\infty \frac{m_n}{n!}\,z^n
  $$
  be the formal Laplace transform
  (or exponential moment generating function).
  Taking the formal logarithm we can write this series as
  $$
  \alg{F}(z) = e^{K(z)}
  $$
  where
  $$
  K(z) = \sum_{n=1}^\infty \frac{\cc_n}{n!}\,z^n
  $$
  is the \emph{cumulant generating function} and the numbers
  $\cc_n$ are called the \emph{(classical) cumulants}
  of the random variable $X$.

  If for a partition $\pi=\{\pi_1,\pi_2,\dots,\pi_p\}$
  we put $m_\pi = m_{\abs{\pi_1}}m_{\abs{\pi_2}}\cdots m_{\abs{\pi_p}}$
  and
  $\cc_\pi = \cc_{\abs{\pi_1}}\cc_{\abs{\pi_2}}\cdots \cc_{\abs{\pi_p}}$,
  then we can express moments and cumulants mutually 
  using the M\"obius function $\mu$ on the partition lattice $\Pi_n$
  as follows:
  $$
  m_\pi = \sum_{\sigma\leq\pi} \cc_\sigma
  \qquad
  \cc_\pi = \sum_{\sigma\leq\pi} m_\sigma \, \mu(\sigma,\pi)
  .
  $$
  For example, the standard Gaussian distribution $\gamma=N(0,1)$
  has cumulants
  $$
  \cc_n(\gamma)
  = \begin{cases}
      1 & n=2\\
      0 & n\ne 2
    \end{cases}
  $$
  It follows that the even moments $m_{2n} = \frac{2n!}{2^n n!}$
  of the standard gaussian distribution
  count the number of pairings of a set with the corresponding number of elements.

\subsection{Free Cumulants}
  Free cumulants were introduced by Speicher
  \cite{RS1} in his combinatorial
  approach to Voiculescu's free probability theory.
  Given our random variable $X$, let
  \begin{equation}
    \label{eq:ordinarymgf}
    M(z) = 1 + \sum_{n=1}^\infty m_n z^n
  \end{equation}
  be its ordinary moment generating function.
  Define a formal power series
  $$
  C(z) = 1 + \sum_{n=1}^\infty \fc_n z^n
  $$
  implicitly by the equation
  $$
  C(z) = C(zM(z))
  .
  $$
  Then the coefficients $\fc_n$ are called the \emph{free} or \emph{non-crossing cumulants}.
  The latter name stems from the fact that
  combinatorially these cumulants are obtained by M\"obius inversion on the lattice of
  non-crossing partitions:
  \begin{equation}
    \label{eq:NCcumulantsMoebiusinversion}
    m_\pi = \sum_{\substack{\sigma\in\NC_n\\ \sigma\leq\pi}} \fc_\sigma
    \qquad
    \fc_\pi = \sum_{\substack{\sigma\in\NC_n\\ \sigma\leq\pi}} m_\sigma \,\mu_{\NC}(\sigma,\pi)
  \end{equation}

\subsection{Boolean cumulants}
Boolean cumulants linearize \emph{boolean convolution}
\cite{SW}.
Let again $M(z)$ be the ordinary moment generating function of a random variable $X$
defined by \eqref{eq:ordinarymgf}.
It can be written as
$$
M(z) = \frac{1}{1-H(z)}
$$
where
$$
H(z) = \sum_{n=1}^\infty \bc_n z^n
$$
and the coefficients are called \emph{boolean cumulants}.
Combinatorially the connection between moments and boolean cumulants is described
by M\"obius inversion on the lattice of interval partitions:
\begin{equation}
  \label{eq:booleancumulantsMoebiusinversion}
  m_\pi = \sum_{\substack{\sigma\in\intpart_n\\ \sigma\leq\pi}}
           \bc_\sigma
  \qquad
  \bc_\pi = \sum_{\substack{\sigma\in\intpart_n\\ \sigma\leq\pi}}
           m_\sigma \,\mu_{\intpart}(\sigma,\pi)
\end{equation}
The connection between these kinds of cumulants is provided by the following
theorem (see also \cite{BSp2} for the case of pairings).
\begin{theorem}[{\cite{Leh}}]
  \label{thm:connectedcumulants}
  Let $(m_n)$ be a (formal) moment sequence with classical cumulants $\cc_n$.
  Then the free cumulants of $m_n$ are equal to
  \begin{equation}
    \label{eq:freeintermsofclassical}
    \fc_n = \sum_{\pi\in \SPconn_{\!\!n}} \cc_\pi
  \end{equation}
  the boolean cumulants are equal to
  $$
  \bc_n = \sum_{\pi\in \SPirr_{\!\!n}} \cc_\pi = \sum_{\pi\in \NCirr_n} \fc_\pi
  $$
\end{theorem}

\subsection{The normal law and pair partitions}
Let $d\gamma(t)=\frac{1}{\sqrt{2\pi}}\,e^{-\frac{t^2}{2}}dt$ be the standard normal
(classical Gaussian) distribution. Its moments are given by
$$m_k:=\int_\RR t^k d\gamma(t)=\begin{cases}
(k-1)!!:=(k-1)(k-3)(k-5)\cdots 5\cdot 3\cdot 1,&\text{$k$ even}\\
0,& \text{$k$ odd}
\end{cases}$$
A more combinatorial description of this is that the moments count all pairings, i.e.,
$$m_k=\#\SP_{\!\!2}(k).$$
From
Theorem~\ref{thm:connectedcumulants}
it follows then that the free cumulants
$\fc_n$ of $\gamma$ are given by the number of connected (or irreducible) pairings,
$$\fc_n=\#\{\pi\in \SP_{\!\!2}(n)\mid\text{$\pi$ is connected}\}.$$
The question whether $\gamma$ is infinitely divisible in the free sense is equivalent to
the question whether the sequence $(\fc_n)_{n\in\NN}$ is conditionally positive, which is
the same as the question whether the shifted sequence $(s_n)_{n\geq 0}$, where
$s_n:=\fc_{n+2}$, is positive definite, i.e., the moment sequence of some measure. See
Section 13 of \cite{NSbook} for more details on this (note that there only compactly
supported measures are considered, but the theory also extends to measures
which 
have a uniquely solvable moment problem).

The first few values of the free cumulants of the Gaussian distribution are
$$\fc_2=1,\quad\fc_4=1,\quad\fc_6=4,\quad\fc_8=27,\quad\fc_{10}=248, \quad\fc_{12}=2830,\quad\dots$$
This sequence of the numbers of irreducible diagrams of $2n$ nodes has been well-studied
from a combinatorial point of view, see, e.g., \cite{Stein1,Stein2}; it appears for
example as sequence A000699 in Sloane's \emph{Encyclopedia of Integer Sequences} \cite{Sloane:www}.
For a recent bibliography of this sequence see
\cite{Klazar:2003:nonPrecursiveness} where it is shown that
the sequence is not holonomic, i.e.,
it does not satisfy a linear recurrence with polynomial coefficients.
However, positivity questions for
this sequence have never been considered. It might be interesting to point out that these
numbers appear also in the perturbation expansion in quantum field theory for the spinor
case in 4 spacetime dimensions, see \cite{BK}
and in renormalization of quantum electrodynamics,
see \cite{Brouder:2000:trees};
however, due to the cryptic style of the mentioned papers
the meaning of this remains quite mysterious for the present authors.

\subsection{A recursive formula}
A main result on the numbers of irreducible diagrams is the following recursion formula
due to Riordan \cite{Riordan:1975:distribution}
$$\fc_{2n}=(n-1)\sum_{i=1}^{n-1} \fc_{2i} \fc_{2(n-i)},$$
a simple bijective proof of which can be found in \cite{NijenhuisWilf:1979:enumeration}.
In terms of the shifted sequence $(s_n)_{n\geq 0}$ this reads
\begin{equation}\label{eq:recursion1}
s_{2n}=n\sum_{i=0}^{n-1} s_{2i}s_{2(n-i-1)}.
\end{equation}
Note the similarity to the standard recursion of the Catalan numbers
(just remove the factor $n$ before the sum).
An equivalent formulation is
\begin{equation}\label{eq:recursion2}
s_{2n}=\sum_{i=0}^{n-1} (2i+1) s_{2i} s_{2(n-i-1)}.
\end{equation}
Thus the question is whether the sequence $(s_k)_{k\geq 0}$ -- defined by $s_{2n+1}=0$
($n\in\NN$) and by either of the recursions \eqref{eq:recursion1} or
\eqref{eq:recursion2} and $s_0=1$ -- is the moment sequence of some measure. Since
$s_0=1$, this measure must necessarily be a probability measure.
The most direct way to prove this would be to find a selfadjoint operator which has these
numbers $s_n$ as moments.

There are some immediate combinatorial interpretations of the above recursions. For
example, the recursion \eqref{eq:recursion2} yields
$$s_{2n}=\sum_{\pi\in NC_2(2n)} \prod_{V\in\pi} (\ip(V)+1).$$
Here we are summing over all non-crossing pairings of $2n$ elements and the contribution
of a pairing $\pi$ is given by a product over the blocks of $\pi$, each block
contributing the number $\ip(V)$ of its inner points plus one.
These inner points have also been counted in
\cite{Yoshida:2002:remarks,EffrosPopa:2003:feynman}
in different contexts.

\subsection{Tree factorials}
The recursion~\eqref{eq:recursion1}, on the other hand,
can be interpreted in terms of planar rooted binary trees.
These are planar rooted trees such that each vertex at most $2$
successors, called \emph{children}. A vertex without successors
is called a \emph{leaf}.
Denote the set of such trees with $n$ vertices by $\PRBT\!_n$. The
number of these trees is the $n$-th Catalan number.
The \emph{tree-factorial} is defined as follows.
For $n=0$ there is only one binary tree (the empty tree),
whose factorial is defined to be $1$.
Let $t$ be a binary tree with $n>0$ vertices.
Then $t$ can be decomposed into its root vertex,
a left branch $t_1$ with $k$ vertices and a right branch
$t_2$ with $n-1-k$ vertices and we define
$$
t! = n\cdot t_1!\, t_2!
$$
Then we have the following identity.
\begin{prop}
\begin{equation}
  \label{eq:s2n}
  s_{2n}=\sum_{t\in \PRBT\!_n} t!
\end{equation}
\end{prop}
Indeed using the above decomposition it is easy to see
that the numbers on the right hand side also satisfy the
recursion~\eqref{eq:recursion1}.
For more information on tree factorials see, e.g., \cite[Section 2]{Mazza}
and \cite[Section~2]{LR} and section~\ref{ssec:Hopf} below.

Note that these interpretations are canonical for the shifted sequence $(s_n)$ and not
for the original sequence $(\fc_n)$; for example, $\fc_8=27$ is the number of irreducible
pairings of 8 points, but $s_6=27$ is given in terms of non-crossing pairings of 6 points
or, equivalently, in terms of planar binary trees with $3=6/2$ nodes.

\subsection{Two Markov chains}
\subsubsection{MTR on binary search trees}
The tree factorial appears in the stationary distribution of the
move-to-root Markov chain on binary trees \cite{DobrowFill:1995:markov}.
Binary trees are used in computer science to arrange data such that
it can be accessed using binary search. To reduce search time,
every time an entry is searched it is moved to the root of the tree
by repeating the so called \emph{simple exchange} shown in the following
picture
$$
\begin{mytikzbox}
  [grow=up,inner sep=1mm]
  \node (root)  {}
    child{ node  [fill,shape=circle]  {} edge from parent [dashed]
      child{node [fill,shape=circle,color=gray]  {}
        child{ node[shape=rectangle,draw] {c} edge from parent [solid]}
        child{ node[shape=rectangle,draw] {b} edge from parent [solid]}
        edge from parent [solid]
      }
      child{ node[shape=rectangle,solid,draw] {a} edge from parent [solid]}
    }
    child[missing]
    ;
  \draw[dashed,->,color=gray,thick](root-1-1) to [bend left=45] (root-1);
  ;
\end{mytikzbox}
\qquad
\to{}
\qquad
\begin{mytikzbox}
  [grow=up,inner sep=1mm]
  \node (root)  {}
    child{node[fill,shape=circle,color=gray]  {} edge from parent [dashed]
      child{ node[shape=rectangle,draw,solid] {c} edge from parent [solid]}
      child{node[fill,shape=circle]  {} edge from parent [solid]
        child{ node[shape=rectangle,draw] {b} edge from parent [solid]}
        child{ node[shape=rectangle,draw] {a} edge from parent [solid]}
      }
    }
    child[missing]
  ;
\end{mytikzbox}
$$
until the root position is reached.
Choosing a vertex randomly (each with probability $1/n$),
this induces a Markov chain on the state space $\PRBT\!_n$.
By Perron-Frobenius theory there is a unique
stationary distribution
$\pi$ for this Markov chain
and it is shown in \cite{DobrowFill:1995:markov}
that it is given by $\pi(t)=1/t!$.
Equivalently, it describes the distribution of a randomly grown tree.

\subsubsection{The Naimi-Trehel algorithm on planar rooted trees}

The tree factorial also appears in the so-called \emph{Naimi-Trehel} algorithm
\cite{NaimiTrehel:1987:algorithme,NaimiTrehelArnold:1996:logN}.
This is a queuing model based on yet another Catalan family,
namely planar rooted trees $\PRT$. It solves a scheduling problem for
$n$ clients (e.g., computers) who access some resource (e.g., a printer)
which can serve at most one client at a time.
In order to reduce the number of messages needed to schedule the
printer jobs,
the queue is arranged as a planar rooted tree and each time a request is
sent, the queue is rearranged. The average number of messages
is then a certain statistic on these trees.
This can be modeled as a Markov chain on labeled rooted trees
where at each step
a random client sends a request and the tree is transformed
accordingly. In the end only the shape of the tree matters
and it suffices to consider the corresponding Markov chain
on the unlabeled planar rooted trees.
By means of bijection this can be transformed into a Markov
chain on Dyck paths where we have the following algebraic
rule for the transition probabilities \cite{NaimiTrehelArnold:1996:logN}.

Let us consider words in the two letter alphabet $\{x,x^*\}$
where $x$ is an \emph{upstep} or \emph{NE} step and
$x^*$ is a \emph{downstep} or \emph{SE} step.
A \emph{Dyck word} is a word in $x$ and $x^*$ such that
each left subword contains not more downsteps than upsteps
and the whole word contains an equal number of up- and downsteps.
Dyck words can be visualized by Dyck paths, see fig.~\ref{fig:dyckpaths}.
These are lattice paths which do not descend below the $x$-axis.
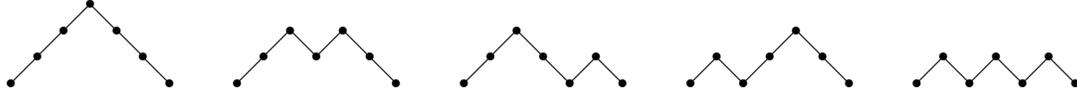
\begin{figure}[ht]
  \centering
\begin{picture}(70,30.0)(1,0)
  \put(0,0){\circle*{3.0}}
  \put(10,10){\circle*{3.0}}
  \put(0,0){\line(1,1){10.0}}
  \put(20,20){\circle*{3.0}}
  \put(10,10){\line(1,1){10.0}}
  \put(30,30){\circle*{3.0}}
  \put(20,20){\line(1,1){10.0}}
  \put(40,20){\circle*{3.0}}
  \put(30,30){\line(1,-1){10.0}}
  \put(50,10){\circle*{3.0}}
  \put(40,20){\line(1,-1){10.0}}
  \put(60,0){\circle*{3.0}}
  \put(50,10){\line(1,-1){10.0}}
\end{picture}
\quad
\begin{picture}(70,20.0)(1,0)
  \put(0,0){\circle*{3.0}}
  \put(10,10){\circle*{3.0}}
  \put(0,0){\line(1,1){10.0}}
  \put(20,20){\circle*{3.0}}
  \put(10,10){\line(1,1){10.0}}
  \put(30,10){\circle*{3.0}}
  \put(20,20){\line(1,-1){10.0}}
  \put(40,20){\circle*{3.0}}
  \put(30,10){\line(1,1){10.0}}
  \put(50,10){\circle*{3.0}}
  \put(40,20){\line(1,-1){10.0}}
  \put(60,0){\circle*{3.0}}
  \put(50,10){\line(1,-1){10.0}}
\end{picture}
\quad
\begin{picture}(70,20.0)(1,0)
  \put(0,0){\circle*{3.0}}
  \put(10,10){\circle*{3.0}}
  \put(0,0){\line(1,1){10.0}}
  \put(20,20){\circle*{3.0}}
  \put(10,10){\line(1,1){10.0}}
  \put(30,10){\circle*{3.0}}
  \put(20,20){\line(1,-1){10.0}}
  \put(40,0){\circle*{3.0}}
  \put(30,10){\line(1,-1){10.0}}
  \put(50,10){\circle*{3.0}}
  \put(40,0){\line(1,1){10.0}}
  \put(60,0){\circle*{3.0}}
  \put(50,10){\line(1,-1){10.0}}
\end{picture}
\quad
\begin{picture}(70,20.0)(1,0)
\put(0,0){\circle*{3.0}}
\put(10,10){\circle*{3.0}}
\put(0,0){\line(1,1){10.0}}
\put(20,0){\circle*{3.0}}
\put(10,10){\line(1,-1){10.0}}
\put(30,10){\circle*{3.0}}
\put(20,0){\line(1,1){10.0}}
\put(40,20){\circle*{3.0}}
\put(30,10){\line(1,1){10.0}}
\put(50,10){\circle*{3.0}}
\put(40,20){\line(1,-1){10.0}}
\put(60,0){\circle*{3.0}}
\put(50,10){\line(1,-1){10.0}}
\end{picture}
\quad
\begin{picture}(70,10.0)(1,0)
  \put(0,0){\circle*{3.0}}
  \put(10,10){\circle*{3.0}}
  \put(0,0){\line(1,1){10.0}}
  \put(20,0){\circle*{3.0}}
  \put(10,10){\line(1,-1){10.0}}
  \put(30,10){\circle*{3.0}}
  \put(20,0){\line(1,1){10.0}}
  \put(40,0){\circle*{3.0}}
  \put(30,10){\line(1,-1){10.0}}
  \put(50,10){\circle*{3.0}}
  \put(40,0){\line(1,1){10.0}}
  \put(60,0){\circle*{3.0}}
  \put(50,10){\line(1,-1){10.0}}
\end{picture}

  \caption{Dyck paths of length $6$.}
  \label{fig:dyckpaths}
\end{figure}
We denote by $1$ the Dyck word of length $0$ and $D_n$ the set
of Dyck words of length $2n$.

The recursive structure of rooted planar binary trees
has a counterpart in the unique decomposition
of a Dyck word $w$ as concatenation
$$
w=uxvx^*
$$
such that both $u$ and $v$ are again Dyck words.
Using this recursive structure
we get a natural bijection $\alpha$ from planar rooted binary trees
to Dyck words by setting recursively
$\alpha(t)=\alpha(t_1)x\alpha(t_2)x^*$
if $t$ has left and right subtrees $t_1$ and $t_2$.
Under this bijection the tree factorial on Dyck words
can be recursively computed as
$$
1! = 1
\qquad
(uxvx^*)! = n\cdot u!\, v!
$$
Dyck words form a monoid with the concatenation product
and following \cite{NaimiTrehelArnold:1996:logN}
we recursively define a linear operator on the monoid
algebra of formal linear combinations of Dyck words
by letting
$$
\mu(w) = w + \nu(w)
$$
where
$$
\nu(1) = 0
\qquad
\nu(uxvx^*) = \nu(u)\owedge(xvx^*) + \mu(v)xux^*
$$
and the operation $\owedge$ is defined as
$$
uxvx^*\owedge w = uxvwx^*
$$
For example, using these rules we have
$$
\mu(\,\begin{picture}(60,20.0)(1,0)
  \put(0,0){\circle*{3.0}}
  \put(10,10){\circle*{3.0}}
  \put(0,0){\line(1,1){10.0}}
  \put(20,20){\circle*{3.0}}
  \put(10,10){\line(1,1){10.0}}
  \put(30,10){\circle*{3.0}}
  \put(20,20){\line(1,-1){10.0}}
  \put(40,20){\circle*{3.0}}
  \put(30,10){\line(1,1){10.0}}
  \put(50,10){\circle*{3.0}}
  \put(40,20){\line(1,-1){10.0}}
  \put(60,0){\circle*{3.0}}
  \put(50,10){\line(1,-1){10.0}}
\end{picture} )
=
\,\begin{picture}(60,20.0)(1,0)
  \put(0,0){\circle*{3.0}}
  \put(10,10){\circle*{3.0}}
  \put(0,0){\line(1,1){10.0}}
  \put(20,20){\circle*{3.0}}
  \put(10,10){\line(1,1){10.0}}
  \put(30,10){\circle*{3.0}}
  \put(20,20){\line(1,-1){10.0}}
  \put(40,20){\circle*{3.0}}
  \put(30,10){\line(1,1){10.0}}
  \put(50,10){\circle*{3.0}}
  \put(40,20){\line(1,-1){10.0}}
  \put(60,0){\circle*{3.0}}
  \put(50,10){\line(1,-1){10.0}}
\end{picture}
+2\,
\,\begin{picture}(60,20.0)(1,0)
  \put(0,0){\circle*{3.0}}
  \put(10,10){\circle*{3.0}}
  \put(0,0){\line(1,1){10.0}}
  \put(20,20){\circle*{3.0}}
  \put(10,10){\line(1,1){10.0}}
  \put(30,10){\circle*{3.0}}
  \put(20,20){\line(1,-1){10.0}}
  \put(40,0){\circle*{3.0}}
  \put(30,10){\line(1,-1){10.0}}
  \put(50,10){\circle*{3.0}}
  \put(40,0){\line(1,1){10.0}}
  \put(60,0){\circle*{3.0}}
  \put(50,10){\line(1,-1){10.0}}
\end{picture}
+\,
\,\begin{picture}(60,10.0)(1,0)
  \put(0,0){\circle*{3.0}}
  \put(10,10){\circle*{3.0}}
  \put(0,0){\line(1,1){10.0}}
  \put(20,0){\circle*{3.0}}
  \put(10,10){\line(1,-1){10.0}}
  \put(30,10){\circle*{3.0}}
  \put(20,0){\line(1,1){10.0}}
  \put(40,0){\circle*{3.0}}
  \put(30,10){\line(1,-1){10.0}}
  \put(50,10){\circle*{3.0}}
  \put(40,0){\line(1,1){10.0}}
  \put(60,0){\circle*{3.0}}
  \put(50,10){\line(1,-1){10.0}}
\end{picture} 
$$
Then it is easy to see by induction that $\mu$ maps
a Dyck path $w$ of length $2n$ to a linear combination
of Dyck paths of the same length and that
the coefficients are nonnegative integers which sum up to $n+1$.
We interpret the matrix representation $A$ of this map
as a weighted adjacency matrix and obtain a digraph
with vertex set $D_n$.
Dividing the matrix by $n+1$ we obtain a stochastic
matrix $P=\frac{1}{n+1}A$.
It was shown in \cite{NaimiTrehelArnold:1996:logN} that
this is exactly the transition matrix of the Naimi-Trehel Markov chain
discussed above.
Figures~\ref{fig:graphs23} and \ref{fig:graphs4} show
the graphs for $n=2,3$ and $4$.
\begin{figure}[ht]
  \centering
    \includegraphics[width=0.15\textwidth, height=.4\textheight, keepaspectratio]{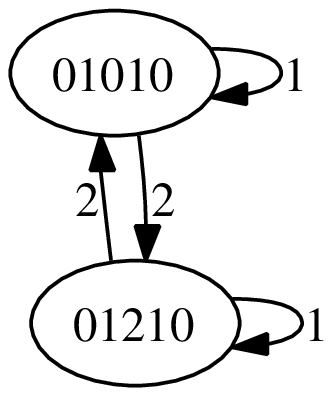}
  \includegraphics[width=0.5\textwidth, height=.4\textheight, keepaspectratio]{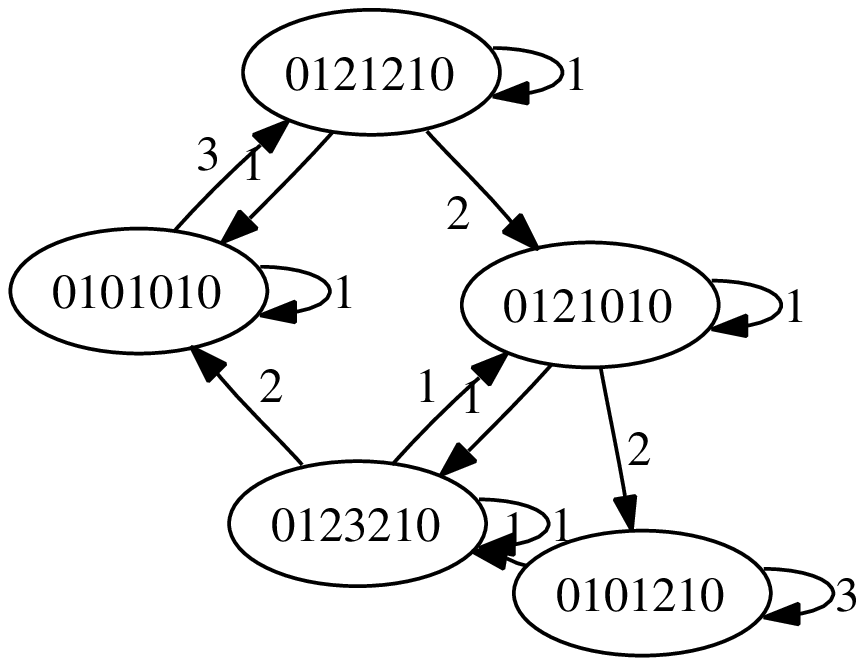}
  \caption{Naimi-Trehel graphs of order $2$ and $3$.}
  \label{fig:graphs23}
\end{figure}
\begin{figure}[ht]
  \centering
  \includegraphics[width=\textwidth, height=.4\textheight, keepaspectratio]{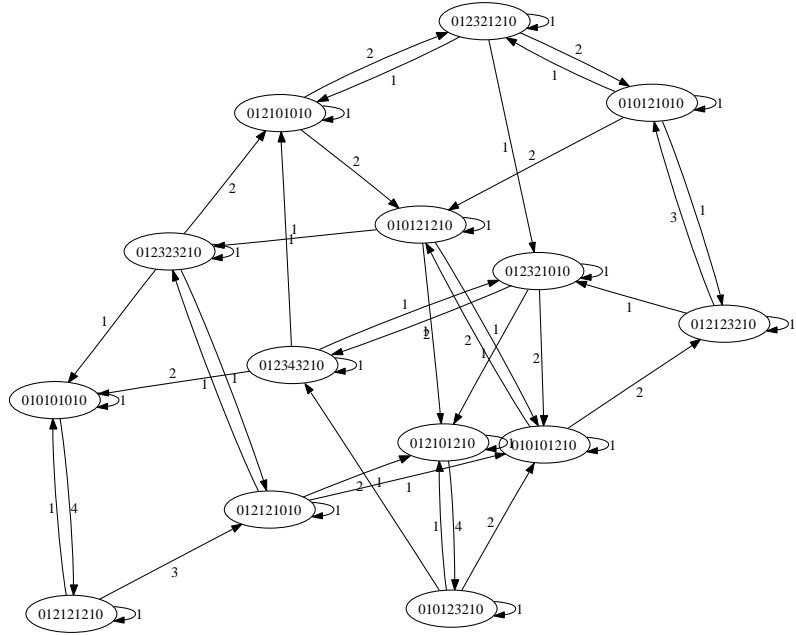}
  \caption{Naimi-Trehel graph of order $4$.}
  \label{fig:graphs4}
\end{figure}
Again there is a unique
stationary distribution
$\pi$ for this Markov chain.
It is shown in  \cite{NaimiTrehelArnold:1996:logN} that $\pi(w)=\frac{1}{w!}$
where $w!$ is the tree factorial defined above.

\subsubsection{A probabilistic interpretation}
For both Markov chains discussed above
our sequence appears as
\begin{equation}
  \label{eq:s2n=sumDn}
  s_{2n}=\sum_{w\in D_n}\frac{1}{\pi(w)}
  = \sum \IE_w T_w
  = C_n \IE T
\end{equation}
where by standard Markov chain theory
\begin{equation}
  \label{eq:1/pi=IET}
  \IE_w T_w=  \frac{1}{\pi(w)}
\end{equation}
is the expected time of a random walker starting
in $w$ to come back to $w$ for the first time
and $\IE T$ is the expected return time of a random walker
starting at a randomly chosen state $w$,
each chosen with probability $1/C_n$ (Catalan number).
Although this setting is very close to the Lindstr\"om-Gessel-Viennot
theory of determinants,
we did not manage to exploit it for a combinatorial proof of our theorem.

\subsection{Loday-Ronco Hopf algebra}
\label{ssec:Hopf}
Hopf algebras of trees have enjoyed increasing interest recently in
renormalization theory and noncommutative geometry
\cite {ConnesKreimer:1998:hopf} and pure algebra (``dendriform algebras'')
\cite{LR}. Hopf algebras of labeled trees have been studied by
Foissy \cite{Foissy:2002:algebres1,Foissy:2002:algebres2}.

We provide here yet another Hopf algebra on labeled trees
whose Hilbert series is related to our problem.
\begin{definition}
  Let $t$ be a planar rooted binary tree.
  A \emph{labeling} of $t$ is a function from
  the vertices of $t$ to the integers.
  A labeling is called \emph{\LRmonotone}
  if the labels are distinct and for every vertex $v$
  of $t$, the labels of the left subtree (with root $v$) are strictly smaller
  than the labels of the right subtree.
  In other words, if we interpret the tree as the Hasse diagram
  of a poset, every antichain has increasing labels.
  Two trees with \LRmonotone{} labelings are called \emph{equivalent}
  if the induce   the same linear order on the vertices.
  The equivalence classes are called \emph{\LRordered{} trees}.
\end{definition}

\begin{prop}
  Let $t$ be a planar rooted binary tree with $n$ vertices.
  Then the tree factorial $t!$ counts the number of
  \LRmonotone{} orderings or equivalently the
  \LRmonotone{} labelings with different numbers $1,2,\dots,n$.
\end{prop}
The proof is a simple induction using the recursive definition
of the tree factorial.

Figure~\ref{fig:LRmonotone} shows an example of a \LRmonotone{} tree.
\begin{figure}[ht]
  \centering
\begin{tikzpicture}[every text node part/.style = {font=\normalsize},  level distance=4ex]
  \node (g3) [vertex] {$3$}
       child{node (g7)[vertex] {$7$}
         child{node (g6)[vertex] {$6$} }
         child{node (g5)[vertex] {$5$} }
       }
       child{node (g2)[vertex] {$2$}
         child{node (g4)[vertex] {$4$} }
         child{node (g1)[vertex] {$1$} }
       }
  ;
\end{tikzpicture}
  \caption{An \LRordered{} tree}
  \label{fig:LRmonotone}
\end{figure}
The formal linear combinations
of labeled planar rooted binary trees form a graded vector space, 
the grading being given by the
number of vertices of the trees.
It is then straightforward to generalize the coproduct of
Loday and Ronco~\cite{LR}
to labeled trees as follows \cite{Foissy:2002:algebres1,Foissy:2002:algebres2}.
Let $s$ and $t$ be labeled binary trees.
We define a new labeled binary tree $s\vee_k t$
by grafting them on a new root with label $k$:
$$
\begin{tikzpicture}[every text node part/.style = {font=\normalsize}]
  \node (k) [vertex] {$k$}
       child{node (t) [vertex] {$t$}
         child{node (t1) [emptyvertex] {} edge from parent [dashed]}
         child{node (t2) [emptyvertex] {} edge from parent [dashed]}
         }
       child{node (s)[vertex] {$s$}
         child{node (s1) [emptyvertex] {} edge from parent [dashed]}
         child{node (s2) [emptyvertex] {} edge from parent [dashed]}
       }
  ;
\end{tikzpicture}
$$

Contrary to the Butcher-Connes-Kreimer Hopf algebra of rooted trees,
binary trees cannot be grafted from forests.
Therefore the left part of the Loday-Ronco coproduct does not consist
of forests, but rather a certain noncommutative ``dendriform''
product of binary trees.
The product of two labeled binary trees $s=s_1\vee_{k} s_2$
and $t=t_1\vee_{l} t_2$
is recursively defined as
\begin{equation}
  \label{eq:LodayRoncoproduct}
  s*t = s_1\vee_{k}(s_2*t) + (s*t_1)\vee_l t_2
\end{equation}
with the convention that for the empty tree $\ET$ the product is
$$
\ET*t = t*\ET = t.
$$
This operation is associative and
the coproduct of $t=u \vee_k v$ is recursively defined as
$$
\Delta(t) = \Delta(u)\ovee_k\Delta(v) + t\otimes\ET
$$
where in Sweedler's notation we define
for $\Delta(u) = \sum u_{(1)}\otimes u_{(2)}$
and $\Delta(v) = \sum v_{(1')}\otimes v_{(2')}$
$$
\Delta(u)\ovee_k\Delta(v) =
\sum\sum (u_{(1)}*v_{(1')})\otimes (u_{(2)}\vee_k v_{(2')})
.
$$
\begin{prop}
  \begin{enumerate}
  \item
    Let $s$ and $t$ be \LRmonotonelabeled{} trees such that the labels
    of $s$ are smaller than the labels of $t$.
    Then $s*t$ is a sum of \LRmonotonelabeled{} trees.
  \item
    Let $t$ be a \LRmonotonelabeled{} tree.
    Then all terms of $\Delta(t)$ contain
    \LRmonotonelabeled{} trees only.
  \end{enumerate}
\end{prop}
\begin{proof}
  \begin{enumerate}
  \item By induction,
    let $s=s_1\vee_k s_2$ and $t=t_1\vee_l t_2$,
    then $s*t=s_1\vee_k(s_2*t) + (s*t_1)\vee_l t_2$.
    By induction hypothesis, the monomials of both $s_2*t$ and $s*t_1$
    are \LRmonotonelabeled{},
    the labels of $s_1$ are smaller than the labels of $s_2*t$
    and the labels of $s*t_1$ are smaller than the labels of $t_2$.
  \item Assume that $t=u\vee_k v$ has \LRmonotone{} labels and its
    children have coproducts
    $\Delta(u) = \sum u_{(1)}\otimes u_{(2)}$
    and $\Delta(v) = \sum v_{(1')}\otimes v_{(2')}$,
    then by the preceding calculation each term $u_{(1)}*v_{(1')}$
    is a sum of \LRmonotonelabeled{} trees and each $u_{(2)}\vee_k v_{(2')}$
    has \LRmonotone{} labels as well, therefore
    $$
    \Delta(t) = \sum\sum (u_{(1)}*v_{(1')}) \otimes (u_{(2)}\vee_k v_{(2')})
    +t\otimes\ET
    $$
    only contains \LRmonotonelabeled{} terms.
  \end{enumerate}
\end{proof}
It follows that the \LRmonotonelabeled{} trees form a graded sub-coalgebra,
but not a subalgebra because the product of \LRmonotonelabeled{} trees
does not consist of \LRmonotonelabeled{} trees in general.
Passing to \LRordered{} trees we obtain a Hopf algebra.
\begin{corollary}
  The \LRordered{} trees span a Hopf algebra whose Hilbert series
  is the generating series of our sequence $(s_n)$.
\end{corollary}
\begin{proof}
  We have seen that \LRmonotonelabeled{} trees form a coalgebra.
  Define the product of \LRordered{} trees $s$ and $t$ as follows:
  Put arbitrary \LRmonotone{} labelings on $s$ and $t$ such that the labels of
  $s$ are smaller than the labels of $t$. Then $s*t$ consists of
  \LRmonotonelabeled{}
  trees
  and the corresponding \LRordered{} equivalence classes
  do not depend on the choice of the labelings for $s$ and $t$.
  Define the coproduct of an \LRordered{} tree by choosing
  an \LRmonotone{} labeling compatible with the \LRorder{},
  compute the coproduct of the obtained \LRmonotonelabeled{} tree and
  replace the \LRmonotonelabeled{}
  terms of the result by the corresponding \LRordered{} trees.
  Again the choice of the initial \LRmonotone{} labeling
  has no influence on the final result;
  moreover, the obtained coalgebra is graded and connected,
  i.e., the first homogeneous component is one-dimensional,
  and the existence of the antipode follows  by standard Hopf algebra theory.
\end{proof}
Some examples:
$$
\begin{aligned}
  \Delta(\ET)
  &=\ET\otimes\ET\\
  \Delta(
\begin{mytikzbox}
  \node (1) [vertex] {$1$};
\end{mytikzbox}
)
  &= \ET\otimes
\begin{mytikzbox}
  \node (1) [vertex] {$1$};
\end{mytikzbox}
  +
\begin{mytikzbox}
  \node (1) [vertex] {$1$};
\end{mytikzbox}
  \otimes\ET \\
%
%
  \Delta(
  \begin{mytikzbox}
    [level 1/.style={sibling distance=1.2em}]
    \node (o) [vertex] {$1$}
        child {node [vertex] {$2$}}
        child [missing] 
    ;
  \end{mytikzbox}
  )
  & = \ET\otimes
  \begin{mytikzbox}
    [level 1/.style={sibling distance=1.2em}]
    \node (o) [vertex] {$1$}
        child {node [vertex] {$2$}}
        child [missing] 
    ;
  \end{mytikzbox}
  +
  \begin{mytikzbox}
    \node (o) [vertex] {$1$}
    ;
  \end{mytikzbox}
  \otimes
  \begin{mytikzbox}
    \node (o) [vertex] {$1$}
    ;
  \end{mytikzbox}
  +
  \begin{mytikzbox}
    [level 1/.style={sibling distance=1.2em}]
    \node (o) [vertex] {$1$}
        child {node [vertex] {$2$}}
        child [missing] 
    ;
  \end{mytikzbox}
  \otimes \ET\\
\end{aligned}
$$

%
%
$$
\begin{aligned}
  \Delta(
  \begin{mytikzbox}
    [level 1/.style={sibling distance=1.2em}]
    \node (o) [vertex] {$3$}
        child {node [vertex] {$2$}}
        child {node [vertex] {$1$}}
    ;
  \end{mytikzbox}
  )
  &=
  \ET
  \otimes
  \begin{mytikzbox}
    [level 1/.style={sibling distance=1.2em}]
    \node (o) [vertex] {$3$}
        child {node [vertex] {$2$}}
        child {node [vertex] {$1$}}
    ;
  \end{mytikzbox}
  +
  \begin{mytikzbox}
    [level 1/.style={sibling distance=1.2em}]
    \node (o) [vertex] {$1$}
    ;
  \end{mytikzbox}
  \otimes
  \begin{mytikzbox}
    [level 1/.style={sibling distance=1.2em}]
    \node (o) [vertex] {$2$}
        child [missing]
        child {node [vertex] {$1$}}
    ;
  \end{mytikzbox}
  +
  \begin{mytikzbox}
    [level 1/.style={sibling distance=1.2em}]
    \node (o) [vertex] {$1$}
    ;
  \end{mytikzbox}
  \otimes
  \begin{mytikzbox}
    [level 1/.style={sibling distance=1.2em}]
    \node (o) [vertex] {$2$}
        child {node [vertex] {$1$}}
        child [missing]
    ;
  \end{mytikzbox}
  +
  \begin{mytikzbox}
    [level 1/.style={sibling distance=1.2em}]
    \node (o) [vertex] {$1$}
        child {node [vertex] {$2$}}
        child [missing]
    ;
  \end{mytikzbox}
  \otimes
  \begin{mytikzbox}
    [level 1/.style={sibling distance=1.2em}]
    \node (o) [vertex] {$1$}
    ;
  \end{mytikzbox}
  +
  \begin{mytikzbox}
    [level 1/.style={sibling distance=1.2em}]
    \node (o) [vertex] {$2$}
        child [missing]
        child {node [vertex] {$1$}}
    ;
  \end{mytikzbox}
  \otimes
  \begin{mytikzbox}
    [level 1/.style={sibling distance=1.2em}]
    \node (o) [vertex] {$1$}
    ;
  \end{mytikzbox}
  +
  \begin{mytikzbox}
    [level 1/.style={sibling distance=1.2em}]
    \node (o) [vertex] {$3$}
        child {node [vertex] {$2$}}
        child {node [vertex] {$1$}}
    ;
  \end{mytikzbox}
  \otimes
  \ET
\end{aligned}
$$

\subsection{Charge Hopf algebra}
There is another coproduct (``charge Hopf algebra'' of Brouder and Frabetti)
in \cite{BrouderFrabetti:2003:QED}, also mentioned in
\cite[sec.~2.6]{Holtkamp:2003:comparison} which is more asymmetric
than the Loday-Ronco Hopf algebra.
First we define an associative multiplication $s/t$ by putting
$\ET/s = s/\ET=s$ and otherwise grafting $s$ onto the leftmost leaf of $t$.
Recursively, $s/(t_1\vee t_2) = (s/t_1)\vee t_2$. This makes sense for labeled
trees as well (just keeping the labels) and for LR-ordered trees
we define $s/t$ by shifting the labels of $t$ such that all labels
of $s$ are less than the labels of $t$ before carrying out the product.
Then the algebra is generated by all elements of the form
$\ET\vee_k t$ which we denote by $V_k(t)$.
The Brouder-Frabetti coproduct is defined recursively as
$$
\begin{aligned}
\DeltaBF(\ET)&=\ET\otimes \ET\\
\DeltaBF(
  \begin{mytikzbox}
    [level 1/.style={sibling distance=1.2em}]
    \node (o) [vertex] {$1$}
    ;
  \end{mytikzbox}
)
&=
  \begin{mytikzbox}
    [level 1/.style={sibling distance=1.2em}]
    \node (o) [vertex] {$1$}
    ;
  \end{mytikzbox}
\otimes
\ET
+
\ET
\otimes
  \begin{mytikzbox}
    [level 1/.style={sibling distance=1.2em}]
    \node (o) [vertex] {$1$}
    ;
  \end{mytikzbox}
\\
\DeltaBF(V_k(s\vee_l t)) &= V_k(s\vee_l t)\otimes\ET + id\otimes V_k(\DeltaBF(s)/(\DeltaBF(V_l(t))-V_l(t)\otimes\ET)
\end{aligned}
$$
Again the monotone trees form a subalgebra.
Examples:
$$
\begin{aligned}
%
%
\DeltaBF(
  \begin{mytikzbox}
    [level 1/.style={sibling distance=1.2em}]
    \node (o) [vertex] {$1$}
        child {node [vertex] {$2$}}
        child [missing]
    ;
  \end{mytikzbox}
)
&= 
  \begin{mytikzbox}
    [level 1/.style={sibling distance=1.2em}]
    \node (o) [vertex] {$1$}
        child {node [vertex] {$2$}}
        child [missing]
    ;
  \end{mytikzbox}
\otimes
\ET
+
\ET
\otimes
  \begin{mytikzbox}
    [level 1/.style={sibling distance=1.2em}]
    \node (o) [vertex] {$1$}
        child {node [vertex] {$2$}}
        child [missing]
    ;
  \end{mytikzbox}
\\
%
%
\DeltaBF(
  \begin{mytikzbox}
    [level 1/.style={sibling distance=1.2em}]
    \node (o) [vertex] {$1$}
        child [missing]
        child {node [vertex] {$2$}}
    ;
  \end{mytikzbox}
) &=
  \begin{mytikzbox}
    [level 1/.style={sibling distance=1.2em}]
    \node (o) [vertex] {$1$}
        child [missing]
        child {node [vertex] {$2$}}
    ;
  \end{mytikzbox}
\otimes\ET
+
2\
  \begin{mytikzbox}
    [level 1/.style={sibling distance=1.2em}]
    \node (o) [vertex] {$1$}
    ;
  \end{mytikzbox}
\otimes
  \begin{mytikzbox}
    [level 1/.style={sibling distance=1.2em}]
    \node (o) [vertex] {$1$}
    ;
  \end{mytikzbox}
+
\ET\otimes
  \begin{mytikzbox}
    [level 1/.style={sibling distance=1.2em}]
    \node (o) [vertex] {$1$}
        child [missing]
        child {node [vertex] {$2$}}
    ;
  \end{mytikzbox}
\\
%
  \DeltaBF(
  \begin{mytikzbox}
    [level 1/.style={sibling distance=2.4em},
    level 2/.style={sibling distance=1em}]
    \node (o) [vertex] {$1$}
        child {node [vertex] {$2$}
          child [missing]
          child {node [vertex] {$3$}}
        }
        child [missing]
    ;
  \end{mytikzbox}
  )
  &=
  \begin{mytikzbox}
    [level 1/.style={sibling distance=2.4em},
    level 2/.style={sibling distance=1em}]
    \node (o) [vertex] {$1$}
        child {node [vertex] {$2$}
          child [missing]
          child {node [vertex] {$3$}}
        }
        child [missing]
    ;
  \end{mytikzbox}
  \otimes\ET
  +
  \begin{mytikzbox}
    [level 1/.style={sibling distance=1.2em}]
    \node (o) [vertex] {$1$}
    ;
  \end{mytikzbox}
\otimes
  \begin{mytikzbox}
    [level 1/.style={sibling distance=1.2em}]
    \node (o) [vertex] {$1$}
        child {node [vertex] {$2$}}
        child [missing]
    ;
  \end{mytikzbox}
  +
  \ET\otimes
  \begin{mytikzbox}
    [level 1/.style={sibling distance=2.4em},
    level 2/.style={sibling distance=1em}]
    \node (o) [vertex] {$1$}
        child {node [vertex] {$2$}
          child [missing]
          child {node [vertex] {$3$}}
        }
        child [missing]
    ;
  \end{mytikzbox}
\end{aligned}
$$

\subsection{Conclusion}
These interpretations of the sequence $s_n$
invite to Fock space like constructions with
corresponding creation and annihilation operators,
with the hope that the sum of creation
and annihilation operator would have the $s_n$ as its moments.
However, we were not able to implement this idea successfully.
The involved inner products usually lacked
positivity.

In the next section we will give an analytic proof of the positive definiteness of the
sequence $s_n$. Essentially, it will consist in showing that the generating power series
$$\phi(z):=\sum_{n=0}^\infty s_{n} \frac 1{z^{n+1}}$$
is actually the Cauchy-transform of a probability measure. We are not able to determine
this measure directly, but we will show its existence. Note that the recursion
\eqref{eq:recursion2} implies, at least formally, for $\phi$ the equation
$$-\phi(z)\phi'(z)=\phi(z)-\frac 1z.$$
This equation will be the starting point of our investigations in the next section and we
will show that it allows an extension of $\phi$ as an analytic map from the upper to the
lower complex half-plane, which is a characterizing property for Cauchy transforms.

Since the $s_n$ grow of the same order as the moments of the
Gauss, the above formal series has no non-trivial radius of convergence. However, it does determine uniquely an analytic
map on some truncated cone at infinity; our proof will show that
this map extends to the upper half-plane.

\section{Analytic proof of the theorem}
In this section we shall give an analytic proof of the free infinite
divisibility of the classical normal distribution. We shall obtain the free infinite divisibility of the classical Gaussian
as a limiting case of a more general family of freely
infinitely divisible distributions with noncompact support, namely
the 
{\em Askey-Wimp-Kerov distributions}, defined in the introduction. 
A certain sub-family of the Askey-Wimp-Kerov distributions
appears in \cite{Bozejko-Guta}. 
Recall that the distributions $\{\mu_c\colon c\in(-1,+\infty)\}$ 
are determined by the continuous fraction 
expansion of their  
Cauchy-Stieltjes transforms:
\begin{equation}\label{C}
G_{\mu_c}(z)=\frac{1}{z-\displaystyle\frac{c+1}{z-\displaystyle\frac{c+2}{z-\ddots}}}.
\end{equation}
In this section, we shall prove the following
theorem:
\begin{theorem}\label{main}
With the notations from above, the probability measures
$\mu_c$ are freely infinitely divisible for all $c\in[-1,0]$.
\end{theorem}
To prove this result, we shall use the well-known characterization
of free infinite divisibility provided by Bercovici and Voiculescu.
Recall \cite{BVIUMJ} that
\begin{theorem}\label{bvid}
A Borel probability measure $\mu$ on the real line is $\boxplus$-infinitely divisible if and only if
its Voiculescu transform $\phi_\mu(z)$ extends to an analytic function $\phi_\mu\colon
\mathbb C^+\to\mathbb C^-$.
\end{theorem}
We remind the reader that the Voiculescu transform of a probability
measure $\mu$ is defined by the equality $\phi_\mu(1/z)=G_\mu^{-1}(z)
-1/z$, for $z$ in some Stolz angle in the lower half-plane, with vertex at zero.\footnote{The name ``Voiculescu transform'' for 
$\phi_\mu$ has been introduced in a later paper 
\cite{BercoviciPata:1999:stable} of Bercovici and 
Pata, not in \cite{BVIUMJ}.}
For more details and important properties of this 
transform, we refer to \cite{BVIUMJ}. In particular, this theorem
guarantees that taking weak limits preserves free infinite divisibility.
Our main source for the analysis of the function $G_{\mu_c}$
will be Kerov's work 
\cite{Kerov:1998:interlacing} and the paper \cite{Askey-Wimp}
of Askey and Wimp. It is shown there that
$\mu_c$ is absolutely continuous with respect to the Lebesgue
measure on the real line and 
\begin{equation}\label{mu_c}
\frac{d\mu_c(u)}{du}=\frac{1}{\sqrt{2\pi}\Gamma(c+1)}\frac{1}{|D_{-c}
(iu)|^2},\quad u\in\mathbb R,
\end{equation}
where for $c>0$,
$$
D_{-c}(z)=\frac{e^{-z^2/4}}{\Gamma(c)}\int_0^\infty 
e^{-zx-\frac{x^2}{2}}x^{c-1}\,dx.
$$
This integral representation does not hold for $c\in(-1,0]$,
but the function $D_{-c}(iu)$ is still well defined for $c\in(-1,0)$,
according to a theorem of Askey and Wimp 
\cite[Theorem 8.2.2]{Kerov:1998:interlacing}.
Moreover, as remarked in \cite[Section 4]{Askey-Wimp}, the function
$D_{-c}(z)$ is an entire function of both $c$ and $z$.
\begin{remark}\label{positivity}
\begin{enumerate}
\item
We observe that, for fixed $c>-1$, equation \eqref{mu_c} together 
with the analyticity of $D$ in $c$ and $z$, guarantees that 
$D_{-c}(iu)\neq0$ for all $u\in\mathbb R$. Indeed, a zero of $\mathbb 
R\ni u\mapsto D_{-c}(iu)\in\mathbb C$ would have to be of order at 
least one, hence the function $\frac{1}{|D_{-c}(iu)|^2}$ would not be 
integrable around that particular zero.
\item On the other hand, analyticity of $D_{-c}$ alone guarantees 
that the density $\frac{d\mu_c(u)}{du}$ is strictly positive 
everywhere on the real line.
\end{enumerate}
\end{remark}
In the same paper of Kerov \cite{Kerov:1998:interlacing}
it is shown that $G_{\mu_c}(z)$
satisfies the Riccati equation
\begin{equation}\label{Riccati}
G_{\mu_c}'(z)=cG_{\mu_c}(z)^2-zG_{\mu_c}(z)+1,\quad z\in\mathbb C^+.
\end{equation}
It is also shown that this expression is equivalent, via the 
substitution
$$
G_{\mu_c}(z)=-\frac1c\frac{\varphi'(z)}{\varphi(z)},
$$
which holds for any $c\neq0$, to the second order linear 
differential equation 
\begin{equation}\label{secondorder}
\varphi''(z)+z\varphi'(z)+c\varphi(z)=0,\quad z\in\mathbb C^+.
\end{equation}
(The function $\varphi$ does depend on $c$.)
The density $\frac{d\mu_c}{du}$ is analytic around zero, so, according
to \cite[Lemma 2.11]{BBG}, it follows that the function $G_{\mu_c}$ has 
an analytic extension to a small enough neighbourhood of zero.
Clearly $G_{\mu_c}(0)\neq0$, so $\varphi$ has an analytic extension 
around zero. Using \eqref{secondorder}, we obtain a convergent power 
series expansion for $\varphi$, namely
$\varphi(z)=\sum_{n=0}^\infty c_nz^n$, where $c_{n+2}=-\frac{c+n}{(n+1)
(n+2)}c_n.$ We conclude that
$$
c_{2k}=(-1)^k\frac{c(c+2)(c+4)\cdots(c+2k)}{(2k)!}c_0,
$$
and
$$
c_{2k+1}=(-1)^k\frac{(c+1)(c+3)\cdots(c+2k-1)}{(2k+1)!}c_1.
$$
Trivially, these coefficients provide a power series with an infinite
radius of convergence. Thus, $\varphi$ is an entire function.
We immediately conclude that $G_{\mu_c}$ extends to a meromorphic 
function defined on all of $\mathbb C$, whose poles coincide with the 
zeros of $\varphi$. We shall denote this extension also by $G_{\mu_c}$. (It may be worth noting that $\varphi$ is entire for any $c\in
\mathbb C$, and that when $c\to-1$, we have $\varphi(z)\to z$, hence
$\mu_c$ tends in the weak topology to the Dirac point mass at zero, $
\delta_0$.)
We shall denote
\begin{equation}\label{notF}
F_{\mu_c}(z)=\frac{1}{G_{\mu_c}(z)},\quad z\in\mathbb C.
\end{equation}
Again, from the above it is clear that $F_{\mu_c}$ is meromorphic,
its poles coinciding with the critical points of $\varphi$.
It satisfies the differential equation
\begin{equation}\label{diffeqF}
F_{\mu_c}'(z)=-F_{\mu_c}(z)^2+zF_{\mu_c}(z)-c,\quad z\in\mathbb C.
\end{equation}
(This follows easily when we divide by $-G_{\mu_c}(z)^2$ in equation
\eqref{Riccati}.) 
Dividing equation 
\eqref{diffeqF} by $F_{\mu_c}(z)$ gives
$$
\frac{F_{\mu_c}'(z)}{F_{\mu_c}(z)}=(z-F_{\mu_c}(z))-\frac{c}{F_{\mu_c}
(z)}.
$$
\begin{proof} (of Theorem \ref{main}.)
We will prove our theorem by arguing that, essentially,
the meromorphic function $F_{\mu_c}$ maps some simply connected 
domain containing the upper half-plane bijectively onto $\mathbb C^+
$. This, together with Theorem \ref{bvid}, will allow us to conclude.
For further reference, we shall split the proof in a succession of
remarks and lemmas. For the rest of the proof, we shall fix
a $c\in(-1,0)$ (open interval!)
Let us start by studying the critical points of $F_{\mu_c}$. 
It is known \cite[Chapter III]{akhieser} that $\Im F_{\mu_c}(z)>\Im z$
for all $z\in\mathbb C^+$. Moreover, if $F_{\mu_c}(z)\in\mathbb C^+$
while $z\not\in\mathbb C^+$, then it is clear that 
$\Im(z-F_{\mu_c}(z))<0.$ In addition, as $c<0$, $\Im\left(-\frac{c}{
F_{\mu_c}(z)}\right)<0$ whenever $F_{\mu_c}(z)\in\mathbb C^+$.
This, together with Remark \ref{positivity}, implies
\begin{remark}\label{+}
If $\Im F_{\mu_c}(z)\geq0$, then $F_{\mu_c}'(z)\neq0$. In 
particular, $F_{\mu_c}$ has no critical points anywhere in the 
closure of the upper half-plane.
\end{remark}
It will be of use to also study the behaviour of $F_{\mu_c}(iy)$,
$y\in\mathbb R$. 
\begin{remark}
The continued fraction expansion of $G_{\mu_c}$ indicates that
the probability measure $\mu_c$ is symmetric with respect to the 
origin, so in
particular $G_{\mu_c}(i[0,+\infty))\subseteq i(-\infty,0]$.
This, of course, together with the meromorphicity of $F_{\mu_c}$
and $G_{\mu_c}$, requires that $G_{\mu_c}(i\mathbb R),F_{\mu_c}(i
\mathbb R)\subseteq i\mathbb R\cup\{\infty\}$.
\end{remark}
Part of the following lemma will
not be used directly in the proof of the current theorem, but
we find that it is nevertheless worth mentioning these results.
\begin{lemma}\label{?}
\begin{enumerate}
\item There exists $-\infty< q_0<0$ so that $F_{\mu_c}$ maps
$i[q_0,+\infty)$ bijectively onto $i[0,+\infty)$.
\item $F_{\mu_c}(i(-\infty, q_0))$ is a bounded subset of 
$i(-\infty,0).$ In particular, $F_{\mu_c}$ has no poles on the
imaginary line.
\item We have $\Im F_{\mu_c}(iy)>y$ for all $y\in\mathbb R$.
\item The function $F_{\mu_c}$ has a unique (simple) critical point 
$is$ in $i(-\infty,-2\sqrt{-c})$. 
In addition, $\lim_{y\to-\infty}F_{\mu_c}'(iy)=
\lim_{y\to-\infty}i^{-1}F_{\mu_c}(iy)=0$.
\item $F_{\mu_c}'(iy)<1$ for all $y\in\mathbb R$.
\end{enumerate}
\end{lemma}
\begin{proof}
For convenience, we denote $f(r)=i^{-1}F_{\mu_c}(ir)$. Clearly,
from the above remark, $f
\colon\mathbb R\to[-\infty,+\infty]$ is real analytic, with the
exceptions of a possible number of points which are poles
of $F_{\mu_c}$. It is an easy exercise to observe that
all poles of $F_{\mu_c}$ must be simple: indeed, otherwise
it would follow from equation \eqref{Riccati} that $G_{\mu_c}$
is identically equal to zero.
Equation \eqref{diffeqF} is re-written for
$f$ as 
\begin{equation}\label{f(r)}
f'(r)=f(r)^2-rf(r)-c,\quad r\in\mathbb R.
\end{equation}
In particular, $F'_{\mu_c}(ir)=f'(r)\in\mathbb R.$
It is known \cite{akhieser} that $\lim_{y\to+\infty}F'_{\mu_c}(iy)
=1,$ and $f(r)=\Im F_{\mu_c}(ir)>r$ for all $r\in(0,+\infty).$
We claim that in fact this must hold for all $r\in\mathbb R$,
fact which excludes the existence of poles on the imaginary axis
for $F_{\mu_c}$. Indeed, continuity of $f$ requires that for
this inequality to be reversed, there must be a point $y\in\mathbb R$
so that $f(y)=y$. Then, from \eqref{f(r)}, it follows that $f'(y)=
-c\in(0,1)$. But for the real analytic $f(r)$ to cross below the 
first bisector as $r$ moves towards $y$ from the right, we clearly 
must have that $f'(y)\ge1.$ 
Contradiction. This implies also that there is no real point at which 
the limit of $f$ from the right is $-\infty$. That the limit
from the right cannot be $+\infty$ at any real point is trivial:
that would make, according to \eqref{f(r)}, $f'$ tend to $+\infty$,
instead of $-\infty$, at the same point 
when $r$ approaches the point
from the right. This proves half of (2) and all of (3).
Next, the critical points of $f$:
First, $f'(s)=0$ is equivalent to 
$$
f(s)=\frac{s\pm\sqrt{s^2+4c}}{2}.
$$
Since $f(s)$ must be real, this excludes points
$s\in\left(-2\sqrt{-c},2\sqrt{-c}\right)$ (recall that $c\in(-1,0)$).
Assuming such a point $s>0$ exists, we must have, from part (3),
$f(s)>s$.
However, that is clearly impossible, since $c<0$.
Thus, only negative $s$ are possible, and for such an $s\le
-2\sqrt{-c}$, we have 
$$
f(s)=\frac{s\pm\sqrt{s^2+4c}}{2}<0,
$$
or, differently stated, both any critical point and any critical
value of $f$ must be negative.
We shall establish that indeed there exists a unique such $s$
(depending of course on $c$), but in order to do that, we will
first prove part (1) of the lemma.
We claim that there exists $-\infty\leq q_0<0$ so that $f$ maps
$[q_0,+\infty)$ onto $[0,+\infty).$
Indeed, otherwise $f(\mathbb R)\subseteq(0,+\infty)$ and, 
as seen above, then $f'(r)>0$ for all $r\in\mathbb R$.
Lagrange's theorem would require that there exists a sequence
$\{r_n\}_n$ which tends to minus infinity so that $f'(r_n)\to0$. But 
this  is impossible, since it would require
$$
\lim_{n\to\infty}f(r_n)(f(r_n)-r_n)=c<0,
$$
while both terms of the product are nonnegative. So 
$0\in f(\mathbb R)$. Choose the largest point in $f^{-1}(\{0\})$
to be $q_0$. Now it is clear that in addition $f$ maps
$[q_0,+\infty)$ bijectively onto $[0,+\infty)$; since $f'(q_0)=
-c>0$, it follows that, first, $f^{-1}(\{0\})=\{q_0\}$, and second,
that the bijective correspondence extends to
a strictly larger interval. 
We show next that this larger interval
cannot be $\mathbb R$. To do this, first let us assume
towards contradiction that $f'(r)>0$ for all $r\in\mathbb R$.
Then, of course, $\lim_{r\to-\infty}f(r)=d$ exists and belongs
to $[-\infty,0)$. We first assume that
$d\in(-\infty,0)$. Then, by Lagrange's theorem again, we must
be able to find a sequence
$\{r_n\}_n$ which tends to minus infinity so that $f'(r_n)\to0$.
As seen above, then 
$$
\lim_{n\to\infty}f(r_n)(f(r_n)-r_n)=c<0,
$$
or, differently said, either 
$$
\lim_{n\to\infty}f(r_n)=0,\quad \textrm{or} \quad
\lim_{n\to\infty}(f(r_n)-r_n)=0.
$$
The first case cannot happen, since it would require that
there exists a critical point of $f$ in $(-\infty,q_0)$, which
we assumed not to happen. The second case would require that
$$
\lim_{n\to\infty}f(r_n)=-\infty,
$$
which would contradict $d\in(-\infty,0)$ again. We consider then
the situation in which $d=-\infty$. To fulfill this condition, and in
addition to avoid that $f'(s)=0$ for some $s\in(-\infty,q_0)$,
it is necessary that 
$$
r<f(r)<\frac{r-\sqrt{r^2+4c}}{2},
$$
for all $r<-2\sqrt{-c}.$
(The necessity of the first inequality was proved before.) However,
recall that $f(q_0)=0\implies f'(q_0)=-c\in(0,1)$, so there must
be then some point $t\in(-\infty,q_0)$ so that $f'(t)=1$.
This implies 
$$
f(t)=\frac{t\pm\sqrt{t^2+4(c+1)}}{2}\left\{\begin{array}{cc}
>0 & {\textrm{if choosing }}+\\
<t & {\textrm{if choosing }}-
\end{array}\right.
$$
which is an obvious contradiction. This, in addition, forbids the 
case of $f'(t)=1$ for any $t\in(-\infty,0]$.
Thus, there must be a critical point $s\in(-\infty,-2\sqrt{-c})$
of $f$. Since by equation \eqref{f(r)} $f'(s)=0\implies
f''(s)=-f(s)>0$, any critical value of $f$ is a local minimum,
hence there exists only one such $s$, and $f(s)$ is the
global minimum of $f$. So $f(r)\in(f(s),0)$ for all $r\in(-\infty,
s)$. The previous arguments about the behaviour of $f$ near
$-\infty$ can be easily reapplied to show that $\lim_{r\to-\infty}
f(r)=\lim_{r\to-\infty}f'(r)=0$. 
Finally, from part (4) above, it follows
that $f'(r)<1$ for all $r\in(-\infty,0]$. Thus, $f'(t)>1$ for some $t
\in\mathbb R$ implies first that $t>0$ and second, that there 
exists a $t_0>0$ so that $f'(t_0)=1$. Differentiating in
\eqref{f(r)} and using part (3) gives
$$
f''(t_0)=f(t_0)-t_0>0.
$$
Thus, $f'$ increases at $t_0$ whenever $f'(t_0)=1$. In particular,
this can happen for only one $t\in\mathbb R$, and thus $f'(r)$ must
tend to one with values strictly greater than one when $r\to+\infty$.
But this contradicts part (3) (namely that $f(r)>r$ for all $r\in
\mathbb R$.) Thus it is impossible to have $f'(t_0)=1$. This proves
(5) and concludes our proof. 
\end{proof}
The following lemma is trivial:
\begin{lemma}\label{surj}
For a fixed $c\in(-1,+\infty)$ there exists a $t>0$ depending on $c$
so that $F_{\mu_c}(\mathbb C^+)\supset\mathbb R+it$.
\end{lemma}
\begin{proof}
Since $F_{\mu_c}$ is analytic, hence open, on the upper half-plane,
and it increases the imaginary part, it is clearly enough to show
that (i) $\lim_{x\to\pm\infty}\Re F_{\mu_c}(x+i)=\pm\infty$ and 
(ii) there exist $M>N\in[0,+\infty)$ so that $N\leq\Im F_{\mu_c}(x+i)
\leq M$ for all $x\in\mathbb R$. 
To prove (i), just observe that, since $\mu_c$ is symmetric and has
all moments, there exists a positive measure $\lambda_c$, also having
all moments, so that $\lambda_c(\mathbb R)=\int_\mathbb Ru^2d\mu_c(u)
$ and
\begin{equation}\label{lambda}
F_{\mu_c}(z)=z-G_{\lambda_c}(z),\quad z\in\mathbb C^+.
\end{equation}
Since $\lambda_c$ is a finite positive measure,
$\lim_{x\to\pm\infty}G_{\lambda_c}(x+i)=0,$ so (i) follows trivially.
Part (ii) is equally simple. We have that 
$$
1<\Im F_{\mu_c}(x+i)=1+\int_{\mathbb R}\frac{1}{1+(x-u)^2}d\lambda_c
(u)\leq 1+\lambda_c(\mathbb R)=1+\int_\mathbb Ru^2d\mu_c(u).
$$
So the lemma is true for any $t\ge1+\int_\mathbb Ru^2d\mu_c(u)$.
\end{proof}
It follows now easily that $F_{\mu_c}$ is injective on the 
upper half-plane and $F_{\mu_c}(\mathbb C^+)$ contains $\mathbb C^+
+it$ for some $t>0$ depending on $c\in(-1,0].$ However,
in order to prove our theorem, we need to find a
larger set $C\supset\mathbb C^+$ which is mapped by $F_{\mu_c}$
bijectively onto the upper half-plane. To do this, we will
show that for {\em any} $t>0$ there exists a $C_t\supset
\mathbb C^++it$ so that $F_{\mu_c}(C_t)=\mathbb C^++it$ and 
$F_{\mu_c}$ is injective on $C_t$. This will clearly
guarantee that $\phi_{\mu_c}$ has an analytic extension
to $\mathbb C^++it$ for any $t>0$, and hence (based on
the previous lemma) a unique extension to $\mathbb C^+$,
concluding the proof of Theorem \ref{main} by an 
application of Theorem \ref{bvid}.
Our strategy will be as follows: for a fixed $t>0$ there
exists, by Lemma \ref{?}, a unique $s>q_0$ so that $t=\frac1i
F_{\mu_c}(is)$. On the other hand, equation \eqref{lambda}
guarantees that there must be a number $N=N(t,c)>0$ so that
$\Im F_{\mu_c}(x+it/2)<t$ for all $x\in\mathbb R$, $|x|>N$.
Since, by Remark \ref{+}, $F_{\mu_c}$ is locally injective
around all these points, we conclude that there are three
simple paths, one around $s$ and two in $\{z\in\mathbb C^+
\colon \Re z>N\}$ and $\{z\in\mathbb C^+
\colon \Re z<-N\}$, respectively, which are mapped by $F_{\mu_c}$
in $\mathbb R+it$. We shall argue that these paths can be
extended to a simple path $p_t$ containing all of them, with
the property that $F_{\mu_c}(p_t)=\mathbb R+it$. The correspondence,
if existing, must be bijective, by Remark \ref{+}, and we will define
$C_t$ to be the simply connected component of $\mathbb C\setminus
p_t$ which contains numbers with arbitrarily large imaginary part.
It will then be easy to prove that $C_t$ has the desired properties
for all $t>0$.
\begin{lemma}\label{cont}
With the above notations, there exists exactly one simple curve 
$p=p_t$, symmetric with respect to the imaginary axis, passing 
through the point $is\in i\mathbb R$ and so that $F_{\mu_c}(p_t)=
\mathbb R+it.$ Moreover, 
$F_{\mu_c}$ maps $p_t$ bijectively onto $\mathbb R+it.$
\end{lemma}
\begin{proof}
For an arbitrary $t>0$,
it is a consequence of Remark \ref{+} and Lemma \ref{?} that 
$F_{\mu_c}$ is conformal on a small enough ball centered at $is$. 
Thus, there exists a simple path $p_t^\varepsilon$, symmetric with respect to the
imaginary axis, which is mapped bijectively by $F_{\mu_c}$ onto some 
interval $(-\varepsilon,\varepsilon)+it$ for $\varepsilon>0$
small enough. We shall show that $p=p_t^\varepsilon$ extends analytically
to an infinite path, denoted by $p_t$, sent by $F_{\mu_c}$ in 
$\mathbb R+it$, on which $F_{\mu_c}$ has no critical points and
so that the limit at infinity of $F_{\mu_c}$ along either half of 
$p_t$ is infinite. Since $F_{\mu_c}$ is meromorphic, hence open, this 
will suffice to prove our lemma. 
Indeed, let us consider the connected component $p^+_t$ of 
$F_{\mu_c}^{-1}([0,+\infty)+it)$ which contains $is$.
It is clear that, as $F_{\mu_c}$ is meromorphic on $\mathbb C$,
the path $p_t^+$ must end either at infinity or at a pole of 
$F_{\mu_c}$, call it $\zeta$. If it ends at a pole, it follows easily 
that  $F_{\mu_c}(p_t^+)=[0,+\infty)+it$ and the correspondence 
(by Remark \ref{+}) is bijective.\footnote{This situation will
turn out later not to occur, but at this moment this is not important for our proof.} 
Let us consider the second case, namely when $p_t^+$ ends at
infinity. In this case, the possibility of having $F_{\mu_c}(p_t^+)=
[0,d)+it$ for some $0<d<+\infty$ must be discarded first:
this would correspond to when $F_{\mu_c}$ has $d+it$ as an asymptotic 
value at infinity along $p_t^+$. Thus, let us show that
$$
\lim_{z\to\infty,z\in p_t^+}F_{\mu_c}(z)=\infty.
$$
Assume towards contradiction that this limit is finite, and call
it $x$ (the case when the limit does not exist is easily discarded).
Of course, $\Im x=t>0$, $\Re x=d>0$. We shall use Equation 
\eqref{diffeqF} to obtain a contradiction: 
it follows from it that
the differential equation satisfied by the inverse $F_{\mu_c}^{-1}$ 
(defined on a neighbourhood of $[0,d)+it$ and with values in a 
neighbourhood of $p_t^+$) is
\begin{equation}\label{inverse}
(F_{\mu_c}^{-1})'(w)=\frac{1}{-w^2+wF_{\mu_c}^{-1}(w)-c}.
\end{equation}
Define $r(v)=\Re F_{\mu_c}^{-1}(v+it),
\iota(v)=\Im F_{\mu_c}^{-1}(v+it),$ $0\le v<d$.
Equation \eqref{inverse} translates into
\begin{equation}\label{rc}
r'(v)=\frac{v(r(v)-v)+t(t-\iota(v))-c}{\left[
v(r(v)-v)+t(t-\iota(v))-c\right]^2+
\left[t(r(v)-2v)+v\iota(v)\right]^2}
\end{equation}
\begin{equation}\label{ic}
\iota'(v)=-\frac{t(r(v)-2v)+v\iota(v)}{\left[
v(r(v)-v)+t(t-\iota(v))-c\right]^2+
\left[t(r(v)-2v)+v\iota(v)\right]^2}
\end{equation}
As noted before, $\lim_{v\uparrow d}F_{\mu_c}^{-1}(v+it)=\infty$,
so at least one of $r(v),\iota(v)$ must be unbounded. 
Thus, at least one of $r'(v),\iota'(v)$ must be unbounded.
From equations \eqref{rc} and \eqref{ic} it follows that in order
for any of $r'(v),\iota'(v)$ to be unbounded, it is necessary that
there exists a sequence $\{v_n\}_{n\in\mathbb N}\subset[0,d)$ tending
to $d$ so that 
$$
\lim_{n\to\infty}\left[
v_n(r(v_n)-v_n)+t(t-\iota(v_n))-c\right]^2+
\left[t(r(v_n)-2v_n)+v_n\iota(v_n)\right]^2=0.
$$
As noted in the comments
preceding Remark \ref{+}, $\iota(v_n)\leq t$.
Let us choose a subsequence of $\{v_n\}_{n\in\mathbb N}$, also denoted 
by $v_n$, on which $\iota(v_n)$ converges. 
If it converges to $\ell\in
(-\infty,t]$,
then we know that $r(v_n)\to\infty$ (since $F_{\mu_c}^{-1}(v_n)\to
\infty$).
But then, in order for the above displayed limit to hold, it is
also necessary that $\iota(v_n)$ tend to infinity, a contradiction.
So we must have that both $r(v_n)$ and $\iota(v_n)$ tend to infinity
(plus or minus). Then
$
\lim_{n\to\infty}\left[
v_n(r(v_n)-v_n)+t(t-\iota(v_n))-c\right]^2=0
\implies
\lim_{n\to\infty}v_n\frac{r(v_n)}{\iota(v_n)}=t,$
or equivalently $\lim_{n\to\infty}\frac{r(v_n)}{\iota(v_n)}=
\frac{t}{d}.$ Similarly, 
$\lim_{n\to\infty}\frac{r(v_n)}{\iota(v_n)}=
\lim_{n\to\infty}\frac{r(v_n)-2v_n}{\iota(v_n)}=
-\frac{d}{t}.$
So $\frac{d}{t}=-\frac{t}{d},$ which implies
$d^2=-t^2$, a contradiction.
Thus, it is impossible for $F_{\mu_c}^{-1}(v_n)$ to tend to infinity
when $v_n\to d$, so $F_{\mu_c}$ cannot have a finite asymptotic 
value along $p_t^+$. Since $\mu_c$ is symmetric, this concludes the proof of our lemma.
\end{proof}
By proving the previous lemma, we have also proved that the inverse 
$F_{\mu_c}^{-1}$ of $F_{\mu_c}$ admits an analytic extension around
$i(0,+\infty)$ and around $\mathbb R+it$ for any $t>0$.
We shall argue that all these extensions agree with each other,
and provide us with an analytic map $F_{\mu_c}^{-1}:\mathbb C^+
\to\mathbb C$ which (1) decreases the imaginary part, and (2)
satisfies $\lim_{y\to+\infty}F_{\mu_c}^{-1}(iy)/iy=1$.
Let us denote 
$$
t_0=\inf\{t>0|\mathbb R+ir\subset F_{\mu_c}(\mathbb C^+\cup\mathbb R)
\forall r\ge t\},
$$
and $s_0$ the unique number greater than $q_0$ so that $F_{\mu_c}(i
s_0)=it_0$.
As noted after the proof of Lemma \ref{surj}, it is clear that
$F_{\mu_c}^{-1}:\mathbb C^++it_0\to\mathbb C^+$ satisfies both 
(1) and (2). 
Clearly, this function has, by Remark \ref{+} and Lemma \ref{?}, a unique
analytic continuation to a small enough neighbourhood of 
$i[0,+\infty)$ in $\mathbb C^+$, which  we will still denote by
$F_{\mu_c}^{-1}$.
Now, for an arbitrary $t\in(0,t_0)$, we have proved in Lemma
\ref{cont} that $F_{\mu_c}^{-1}$ admits an analytic continuation
to a small enough neighbourhood of $\mathbb R+it$ which coincides
on a neighbourhood of $it$ with $F_{\mu_c}^{-1}.$ Since
$\mathbb C^+$ is simply connected and $\displaystyle\mathbb C^+=
(\mathbb C^++it_0)\cup\bigcup_{t\in(0,t_0)}(\mathbb R+it)$,
we conclude that $F_{\mu_c}^{-1}$ admits a unique extension to
the upper half-plane, with values in $\mathbb C$.
Let us denote $C=F_{\mu_c}^{-1}(\mathbb C^+)$.
It follows easily from the identity principle for analytic functions
that $F^{-1}_{\mu_c}(F_{\mu_c}(z))=z$ for all $z\in C$ and
$F_{\mu_c}(F^{-1}_{\mu_c}(z))=z$ for all $z\in\mathbb C^+$.
Moreover, $\Im F_{\mu_c}(z)>\Im z$ for all $z\in\mathbb C^+\cup
\mathbb R$, and if $z\in\mathbb C^-$, $F_{\mu_c}(z)\in\mathbb C^+,$ 
then it is obvious that the inequality $\Im F_{\mu_c}(z)>\Im z$ still 
holds. Thus, we conclude that both (1) and (2) are satisfied by
$F_{\mu_c}^{-1}.$
Now our theorem  follows: we define $\phi_{\mu_c}(z)=
F_{\mu_c}^{-1}(z)-z,$ $z\in\mathbb C^+$. This function is obviously
well defined, and maps the upper half-plane into $\mathbb C^-$, thus
satisfying the requirements of Theorem \ref{bvid}.
\end{proof}
Since the classical Gaussian $\gamma$ equals $\mu_0$, we are now able 
to conclude its free infinite divisibility from Theorem \ref{main}
and Theorem \ref{bvid}. 
\begin{corollary}\label{Gauss}
The classical normal distribution $d\gamma(t)=\frac{1}{\sqrt{2\pi}}
e^{-t^2/2}dt$ is freely infinitely divisible.
\end{corollary}
Next we shall discuss some properties of the Askey-Wimp-Kerov 
distributions $\mu_c$ with parameter $c\in(0,1)$.
As mentioned in the introduction, numerical computations show that
$\mu_c$ is not freely infinitely divisible for certain values of $c>0$.
More specific, direct (computer assisted) calculation of the 
Hankel determinants of the free cumulants of $\mu_{9/10}$, for example,
shows that the $97^{\rm th}$ determinant is negative, and thus the
cumulant series of $\mu_{9/10}$ is not the moment sequence of a 
positive measure on $\mathbb R$; Theorem \ref{bvid} allows us then
to conclude that $\mu_{9/10}$ is not freely infinitely divisible.
(More such computations have been performed, and they seem to indicate
that the size of the first Hankel matrix whose determinant is negative rather tends to
decrease as $c>0$ increases: for example, the $83^{\rm th}$ Hankel
determinant corresponding to $\mu_1$ is negative.)

However, the family $\{\mu_c:c\in(0,1]\}$ turns out to be of some
interest from the point of view of the arithmetic properties of free additive
convolution. This subject is not new (implicit results on the arithmetic of free convolutions
can be found in many works), but it is the rather recent preprint 
\cite{ChistyakovGoetze:2005:arithmetic} that has first addressed the problem of the decomposability
of measures in ``free convolution factors'' in an explicit and systematic way. However, the 
subject is by no means exhausted, the number of results is rather small (we would like to
mention among them a remarkable idecomposability result given in 
\cite{BercoviciWang:2008:indecomposable}), so we feel it is worth mentioning the following
by-product of equation \eqref{C} and our main free infinite divisibility result from Theorem \ref{main}.

Indeed, it follows from the continued fraction expansion
\eqref{C} and analytic continuation that for any $c\in(-1,0]$,
\begin{equation}\label{c+1}
G_{\mu_{c+1}}(z)=\frac{z-F_{\mu_c}(z)}{c+1},\quad z\in\mathbb C.
\end{equation}
In addition, for any fixed $c\in(-1,0]$, 
the dilation transformation $\mu_c\mapsto\mu_c^1$ induced by 
$F_{\mu_c^1}(z)=\frac{1}{\sqrt{c+1}}F_{\mu_c}(z\sqrt{c+1})$
provides us with a probability $\mu_c^1$ of variance one, and
thus
\begin{equation}\label{var}
F_{\mu_c^1}(z)=z-\sqrt{c+1}G_{\mu_{c+1}}(z\sqrt{c+1})=
z-G_{\tilde{\mu}_{c+1}}(z),\quad z\in\mathbb C,
\end{equation}
where of course $\tilde{\mu}_{c+1}$ is a probability measure
obtained by a dilation with a factor of $\sqrt{1+c}$ of 
${\mu}_{c+1}$.
It is noted in \cite[Theorems 1.2 and 1.6]{BN} (see also references therein) that a probability
measure $\lambda$ with variance one and first moment zero is freely 
infinitely divisible if and only if there exist two probabilities
$\nu,\rho$ on $\mathbb R$ so that 
\begin{enumerate}
\item[(a)] $F_\rho(z)=z-G_\nu(z)$ for all $z\in\mathbb C^+$,
\item[(b)] $F_{\lambda}(z)=z-G_{\nu\boxplus\mathcal S}(z)$, $z\in
\mathbb C^+$, where $\lambda=\left(\rho^{\boxplus 2}\right)^{\uplus 
1/2}$ and $\mathcal S$ is the centered Wigner (semicircular) 
distribution of variance one. (Operation $\uplus$
is called Boolean convolution - see \cite{SW}.)
\end{enumerate}
We apply this observation to $\lambda=\mu_c^1$ to conclude from
\eqref{var} that the Voiculescu transform $\phi_{\mu_c^1}$ of 
$\mu_c^1$ is also the Cauchy-Stieltjes transform of a probability 
measure $\nu_c$ playing the role of $\nu$ in (a) and (b) above and 
moreover
$$
\tilde{\mu}_{c+1}=\nu_c\boxplus\mathcal S.
$$
This provides us with another interesting decomposition
result in the arithmetic of free additive convolution, stating that 
\begin{remark}
For any $c\in(-1,0]$, the Askey-Wimp-Kerov distribution
$\mu_{c+1}$ can be written as a free additive convolution of the 
Wigner law with another probability $\nu_c$ on $\mathbb R$.

\end{remark}
\bibliography{Gauss}
\bibliographystyle{plain}
\end{document}